\documentclass{article}
\usepackage{arxiv}

\usepackage[utf8]{inputenc} 
\usepackage[T1]{fontenc}    
\usepackage{hyperref}       
\usepackage{url}            
\usepackage{booktabs}       
\usepackage{nicefrac}       
\usepackage[english]{babel}
\usepackage{ae}
\usepackage{eucal}
\usepackage[dvips]{graphicx}
\usepackage{epsfig}
\usepackage{graphicx}
\usepackage{amssymb}
\usepackage{amsthm}
\usepackage{amsfonts,amsmath}
\usepackage{amssymb}
\usepackage{a4}
\usepackage{apu}
\usepackage{xcolor}
\usepackage{amsmath}
\usepackage{mathtools}
\usepackage{multicol}
\usepackage[all]{xy}

\usepackage{dcolumn,amsthm}

\renewenvironment{proof}[1][Proof]{\noindent\textit{#1. } }{\hfill$\square$}

\newtheoremstyle{theorem}{6pt}{6pt}{\rm}{}{\sffamily}{ }{ }{}
\theoremstyle{theorem}
\newtheorem{theorem}{\sc Theorem}[section]

\newtheoremstyle{lemma}{6pt}{6pt}{\rm}{}{\sffamily}{ }{ }{}
\theoremstyle{lemma}
\newtheorem{lemma}{\sc Lemma}[section]

\newtheoremstyle{example}{6pt}{6pt}{\rm}{}{\sffamily}{ }{ }{}
\theoremstyle{example}

\newtheoremstyle{corollary}{6pt}{6pt}{\rm}{}{\sffamily}{ }{ }{}
\theoremstyle{corollary}

\newtheoremstyle{definition}{6pt}{6pt}{\rm}{}{\sffamily}{ }{ }{}
\theoremstyle{definition}
\newtheorem{definition}[theorem]{\sc Definition}

\newtheoremstyle{remark}{6pt}{6pt}{\rm}{}{\sffamily}{ }{ }{}
\theoremstyle{remark}

\newtheoremstyle{approximation}{6pt}{6pt}{\rm}{}{\sffamily}{ }{ }{}
\theoremstyle{approximation}

\newtheoremstyle{scheme}{6pt}{6pt}{\rm}{}{\sffamily}{ }{ }{}
\theoremstyle{scheme}

\usepackage{biblatex}
\usepackage{csquotes}

\title{Navier-Stokes equations on Riemannian manifolds}

\author{
 Maryam Samavaki \\
  Department of Physics and Mathematics\\
  University of Eastern  Finland\\
  P.O. Box 111, FI-80101 Joensuu,Finland  \\
  \texttt{maryam.samavaki@uef.fi} \\
  \AND
  Jukka Tuomela \\
  Department of Physics and Mathematics\\
  University of Eastern  Finland\\
  P.O. Box 111, FI-80101 Joensuu,Finland  \\
  \texttt{jukka.tuomela@uef.fi} \\
}

\begin{document}
\maketitle

\begin{abstract}
We study properties of the solutions to Navier-Stokes system on compact Riemannian manifolds. The motivation for such a formulation comes from atmospheric models as well as some thin film flows on curved surfaces. There are different choices of the diffusion operator which have been used in previous studies, and we make a few comments why the choice adopted below seems to us the correct one.  This choice leads to the conclusion that Killing vector fields are essential in analyzing the qualitative properties of the flow.  We give several results illustrating this and analyze also the linearized version of Navier-Stokes system which is interesting in numerical applications. Finally we consider the 2 dimensional case which has specific characteristics, and treat also the Coriolis effect which is essential in atmospheric flows.
\end{abstract}



\section{Introduction}
\label{intro}
Navier Stokes equations have been widely studied both form theoretical and applied points of view \cite{temam-ns}. In recent years there has been a growing interest of this system on Riemannian manifolds \cite{carati,chczdi,kobayashi,pierfelice,reuvoi,taylor3} and the many references therein. There seems to be two different reasons for this interest. First are the atmospheric models where the curvature of earth matters if one wants to simulate the flow in very large domains or even on the whole earth. The second are the flows of very thin films on the curved surfaces. Although these two applications are physically very different they both lead naturally to the idea of formulating  the Navier-Stokes equations on arbitrary Riemannian manifolds.

There have been different choices for the diffusion operator for the system on the manifolds, and we discuss first some reasons why we think that the choice adopted below is the appropriate one. The same choice is advocated also in \cite{chczdi,taylor3}. It turns out that this choice of diffusion operator has important consequences on the qualitative and asymptotic properties of the flow, and our choice implies that Killing vector fields are essential in the analysis.

Our main results concern the decomposition of the flows to Killing component and its orthogonal complement. The Killing vector field is actually a solution to the Navier-Stokes system, but due to nonlinearity the orthogonal complement satisfies a different system. However, it is possible to derive similar a priori estimates for this complement than to the total flow.  Interestingly similar conclusion remains valid when one replaces the diffusion operator with another operator and Killing fields with harmonic vector fields. Hence depending on the choice of the diffusion operator the solutions obtained are completely different asymptotically.

We will also analyze the linearized version of Navier-Stokes system. This is interesting at least from the point of view of numerical solution of Navier-Stokes system. Often one uses the idea of operator splitting in order to treat the linear diffusion term and the nonlinear convection term differently (a thorough overview of numerical methods for Navier-Stokes system is given in \cite{glowinski}). Then in some numerical methods one linearizes the convection term to advance the solution. We show that also the linearized version respects this decomposition to Killing fields and the orthogonal complement.

Moreover it turns out that one can produce new solutions with Lie bracket. Given a solution to a linearized system and a Killing field their bracket is also a solution to the linearized system. This is rather a technical result where we show that various differential operators behave well with respect bracket operation when one of the fields is a Killing field.

Since 2 dimensional manifolds are especially important in applications we analyze this special case more closely. In particular the sphere is relevant in meteorological applications so we consider this in detail.  It turns out that one can decompose also the vorticity in the same way as the flow field itself.  Since the vorticities of the Killing fields are simply the first spherical harmonics one can use this to get good a priori estimates. Finally we consider the case of Navier-Stokes on the sphere with the Coriolis term. In this case the Killing field along the latitudes is still a solution and asymptotically the solutions approach it. Note that here again the aymptotic properties depend essentially on the choice of the diffusion operator.

\section{Model and the diffusion operator}
The standard way to write the Navier Stokes equations  in $\mathbb{R}^n$ is as follows
\begin{align*}
  & u_t+u\nabla u-\mu \Delta u+\nabla p=f  \\
  & \nabla \cdot u=0
\end{align*}
Let us formulate this on an arbitrary Riemannian manifold $M$ with Riemannian metric $g$. Let $\nabla$ now denote the covariant derivative, and to avoid confusion we write $\mathsf{grad}(p)=g^{ij}p_{;j}$ for the gradient and  $\mathsf{div}(u)=\mathsf{tr}(\nabla u)=u^i_{;i}$ for the divergence.\footnote{Einstein summation convention is used where needed.} The nonlinear term is now $\nabla_u u=u^k_{;i}u^i$.

The diffusion term $\Delta u$ is more problematic since there are various ways to generalize the Laplacian for vectors. Let us consider some possiblities which are proposed.
 One choice is the \emph{Bochner Laplacian}, defined by the formula
\[
     \Delta_Bu=\mathsf{div}(g^{ij} u^k_{;i})=g^{ij} u^k_{;ij}
\]
This is perhaps mathematically natural, since this is in a sense the first thing that comes to mind, considering that the Laplacian of the scalar function is $\Delta f=g^{ij}f_{;ij}$.  However, apparently there is no physical justification for this choice.

The second is the Hodge Laplacian. This is initially defined for forms, but with the metric we extend it to vector fields. To this end it is convenient to express exterior derivative and its dual in terms of covariant derivatives. In \cite{chowetal} one can find the general formulas, but since the vector field case is sufficient for us let us see only this case. Hodge Laplacian for vector fields is given by formula $\Delta_Hu=-\sharp (\delta d+d\delta)\flat u$. Now obviously
\[
        \sharp d\delta\flat u=\mathsf{grad}(\mathsf{div}(u))=g^{kj} u^i_{;ij}
\]
Then for one form $\alpha$ we have $d\alpha=\alpha_{i;j}-\alpha_{j;i}$ and for a two form $\omega$ we have $\delta \omega=-g^{ij}\omega_{ik;j}$. Let us further define
\begin{equation}
   Au     =g^{ki}u^j_{;i}-g^{ij}u^k_{;i}
\label{harmo}
\end{equation}
Then we can write
\[
   \sharp \delta d\flat u=\mathsf{div}(Au)=g^{ki}u^j_{;ij}-g^{ij}u^k_{;ij}=g^{ki}u^j_{;ij}-\Delta_B u
\]
Now using the Ricci identity \eqref{general-ricci-identity} we obtain
\[
  \Delta_Hu=\mathsf{div}(Au)+\mathsf{grad}(\mathsf{div}(u))=\Delta_Bu-\mathsf{Ri}(u)
\]
where $\mathsf{Ri}$ is the Ricci tensor. In two and three dimensional cases we also have the familiar formulas
\begin{align*}
  \Delta_Hu=&\mathsf{grad}(\mathsf{div}(u))-\mathsf{Rot}(\mathsf{rot}(u))\tag{2D}\\
   \Delta_Hu=&\mathsf{grad}(\mathsf{div}(u))-\mathsf{curl}(\mathsf{curl}(u))\tag{3D}
\end{align*}
where the operators $\mathsf{rot}$, $\mathsf{Rot}$ and $\mathsf{curl}$ are defined in \ref{curl-ja-muut}.

However, one can argue that Hodge Laplacian is not appropriate for the present purposes either. Recall that a (Newtonian) fluid is characterized by the fact that the stress tensor is a function of deformation rate tensor \cite{temmir}. Classically the deformation rate tensor is $\tfrac{1}{2}\,(\nabla u+(\nabla u)^T)$. In the Riemannian case we set (omitting the factor $1/2$)
\[
   Su     =g^{ki}u^j_{;i}+g^{ij}u^k_{;i}
\]
Hence, as in \cite{chczdi,taylor3}, we get the diffusion operator $Lu=\mathsf{div}(Su)$.
\begin{lemma}
\[
  Lu=\Delta_B u+\mathsf{grad}(\mathsf{div}(u))+\mathsf{Ri}(u)
\]
\label{Lu0}
\end{lemma}
\begin{proof}First we compute
\[
Lu= g^{ki}u^j_{;ik}+g^{ij}u^{k}_{;ik}=\Delta_B u+g^{ij}u^{k}_{;ik}
\]
Using the Ricci Identity \eqref{general-ricci-identity}  we get
\[
  g^{ij}u^{k}_{;ik}=g^{ij}u^{k}_{;ki}+g^{ij}u^{\ell}\mathsf{Ri}_{i\ell}
  =\mathsf{grad}(\mathsf{div}(u))+\mathsf{Ri}(u)
\]
\end{proof}

So the operators $L$ and $\Delta_H$ give different signs for the curvature term.  Summarising we may say that Bochner Laplacian uses information about the whole of $\nabla u$ and ignores the curvature term, while $L$ uses the symmetric part, and Hodge Laplacian uses the antisymmetric part. The sign of the  curvature term is different in the symmetric and antisymmetric cases.

It seems to us that the operator $L$ is physically most natural candidate for the diffusion operator because it most naturally generalizes the constitutive laws which are used in the Euclidean spaces. Also in \cite{chczdi} the authors come to the conclusion that $L$ is the best choice.  However,  in \cite{kobayashi} it is argued that $\Delta_H$ should be used, at least in some situations. Also in \cite{carati} $\Delta_H$  is used while in \cite{pierfelice} Bochner Laplacian is used. We do not know how the choice of Bochner Laplacian or the Hodge Laplacian should be interpreted from the point of view of continuum mechanics.

So we take $L$ as our diffusion operator and proceed our analysis with it:
\begin{equation}
\begin{aligned}
  & u_t+\nabla_u u-\mu  L u+\mathsf{grad}( p)=0\\
  &  \mathsf{div}(u)=0
\end{aligned}
\label{ns}
\end{equation}
However, some of the results are valid whatever the choice of the Laplacian and we will analyze what kind of effect this choice has.

\section{Preliminaries and notation}
Let us now introduce some appropriate functional spaces, see  \cite{hebey} for more details. Let us  define the $L^2$ inner product for functions and vector fields by the formulas
\begin{align*}
       \langle f,h\rangle=&\int_M f\,h\,\omega_M       \\
       \langle u,v\rangle=&\int_M g(u,v)\omega_M
\end{align*}
where $\omega_M$ is the volume form (or Riemannian density if $M$ is not orientable). This gives the norm $\|u\|_{L^2}=\sqrt{ \langle u,u\rangle}$. Similarly we can introduce inner products for tensor fields. However, since we need this just for one forms and tensors of type $(1,1)$ we give the formulas only for this case. For one forms $\alpha$ and $\beta$ we can simply write $g(\alpha,\beta)=g(\sharp \alpha,\sharp\beta)=g^{ij}\alpha_i\beta_j$. Then let $T$ be a tensor of type $(1,1)$; pointwise $T$ can be interpreted as a map $T\,:\, T_pM\to T_pM$. Let $T^\ast$ be the adjoint, i.e.
\[
       g(Tu,v)=g(u,T^\ast v)
\]
for all $u$ and $v$. Then the inner product on the fibers can be defined by
\[
    g(T,B)=\mathsf{tr}(TB^\ast)=T^k_\ell g^{j\ell}B^i_j g_{ik}
      =T^{kj}B_{jk}
\]
In this way we can define the familiar Sobolev inner product
\[
  \langle u,v\rangle_{H^1}=\int_M \Big(g(u,v)+ g(\nabla u,\nabla v)\Big)\omega_M
\]
and the corresponding norm. Of course in a similar fashion more general Sobolev spaces can be defined but this is sufficient for our purposes. Finally let us recall that the divergence theorem remains valid in the following form:
\[
     \int_M \mathsf{div}(u)\omega_M=
     \int_{\partial M} g(u,\nu )\omega_{\partial M}
\]
Here $\nu$ is the outer unit normal and $\omega_{\partial M}$ is the volume form (or Riemannian density) induced by $\omega_M$. Note that orientability of $M$ is not needed.  From now on we will always suppose that $M$ is compact and without boundary.

Above we have viewed $S$ and $A$ as  differential operators which operate on $u$. Let us write $S_u$ and $A_u$ when we consider $Su$ and $Au$ as tensors of type $(1,1)$. Note that $S_u=S_u^\ast$ and $A_u=-A_u^\ast$.
The following vector fields are important in the subsequent analysis.
\begin{definition} Vector field $u$ is parallel if $\nabla u=0$, it is Killing if $Su=0$ and it is harmonic, if $\Delta_H u=0$.
\end{definition}
Equivalently we can say that  $u$ is Killing, if
\begin{equation}
  g(\nabla_v u,w)+g(\nabla_w u,v)=0
\label{killing-ehto}
\end{equation}
for all $v$ and $w$. Note that $\mathsf{div}(u)=0$ for Killing fields because $\mathsf{div}(u)=\tfrac{1}{2}\,\mathsf{tr}(S_u)$. If $M$ is compact and without boundary we have the following classical characterization for harmonic vector fields
\[
   \Delta_Hu=0\quad\Leftrightarrow\quad
\begin{cases}
    d\flat u=0 \\ \delta \flat u=0
    \end{cases}\quad\Leftrightarrow\quad
\begin{cases}
    Au=0\\ \mathsf{div}(u)=0
    \end{cases}
\]
where $Au$ is given in \eqref{harmo}. Hence in particular
\[
   g(\nabla_v u,w)-g(\nabla_w u,v)=0
\]
for all $v$ and $w$ if $u$ is harmonic.

There are severe topological restrictions for the existence of parallel vector fields \cite{welsh}.
Killing vector fields and harmonic vector fields are much more common.
Let us recall the following facts \cite{petersen}.
\begin{itemize}
 \item[(i)]If $M$ is $n$ dimensional then the Killing vector fields are a Lie algebra whose dimension is $\le \tfrac{1}{2}\,n(n+1)$ and the equality is attained for the standard sphere.
 \item[(ii)] the space of harmonic vector fields is isomorphic to the first de Rham cohomogy group of $M$. In particular this space is also always finite dimensional.
\end{itemize}

\begin{lemma} Let $u$, $v$  and  $w $ be  vector fields and $\mathsf{div}(v)=0$. Then
\[
   \int_{M}\Big( g(\nabla_v u,w )+g(\nabla_v w  ,u)\Big)\omega_{M}=0
\]
In particular
\[
   \int_{M} g(\nabla_v u,u )\omega_{M}=0
\]
\label{b-lemma}
\end{lemma}
\begin{proof}
Since
\[
  \mathsf{div}\big(g(u,w )\,v\big)=g(u,w )\mathsf{div}(v)+g(\nabla_v u,w )
    +g(\nabla_v w ,u)
\]
the result follows from divergence theorem.
\end{proof}
\begin{lemma}
If $w$ is Killing and $\mathsf{div}(u)=\mathsf{div}(v)=0$  then
\[
    \int_{M}\big(g(\nabla_{v}u,w )+g(\nabla_{u}v,w )\big)\omega_{M}=0
\]
and if $w$ is harmonic and $\mathsf{div}(u)=\mathsf{div}(v)=0$  then
\[
    \int_{M}\big(g(\nabla_{v}u,w )-g(\nabla_{u}v,w )\big)\omega_{M}=0
\]
\label{c-lemma}
\end{lemma}
\begin{proof}
Using Lemma \ref{b-lemma} and formula (2) we obtain
\[
   \int_{M}\big(g(\nabla_{v}u,w )+g(\nabla_{u}v,w )\big)\omega_{M}=-\int_{M}\big(g(\nabla_{v}w,u)+g(\nabla_{u}w,v)\big)\omega_{M} =0
\]
The proof of the second statement is analogous.
\end{proof}

From this we immediately get
\begin{lemma}
If either (i) $u$ is Killing  and $\mathsf{div}(v)=0$ or  (ii) $v$ is Killing and  $\mathsf{div}(u)=0$ then
\[
    \int_{M}g(\nabla_{u}u,v )\omega_{M}=
    \int_{M}g(\nabla_{u}v,u )\omega_{M}=0
\]
and if $u$ is harmonic  and $\mathsf{div}(v)=0$ then
\[
    \int_{M}g(\nabla_{u}u,v )\omega_{M}=0
\]
\label{d-lemma}
\end{lemma}
\begin{proof} Follows directly from previous Lemmas.
\end{proof}
\begin{lemma} Let $u$ and $v$ be  vector fields. Then
\begin{align*}
  &\int_{M}g(\Delta_B u,v)\omega_{M}+\int_{M} g(\nabla u,\nabla v)\omega_{M}=0  \\
  &\int_{M}g(Lu,v)\omega_{M}+\tfrac{1}{2}\int_{M} g(S_u,S_v)\omega_{M}=0
\end{align*}
Hence $\langle \Delta_B u,v\rangle=\langle u, \Delta_B v\rangle$ and $\langle L u,v\rangle=\langle u, Lv\rangle$.
\label{var-lemma}
\end{lemma}
\begin{proof} We compute
\begin{align*}
     \mathsf{div}(g^{ij} u^k_{;i}g_{k\ell}v^\ell)=&g^{ij} u^k_{;ij}g_{k\ell}v^\ell+g^{ij} u^k_{;i}g_{k\ell}v^\ell_{;j}=g(\Delta_B u,v)+g(\nabla u,\nabla v)  \\
     \mathsf{div}(S_u v)=&g^{ki}u^j_{;ik}g_{j\ell}v^\ell+g^{ki}u^j_{;i}g_{j\ell}v^\ell_{;k}+u^k_{;\ell k}v^\ell+u^k_{;\ell}v^\ell_{;k}    \\
       =&\tfrac{1}{2}\,g(S_u,S_v)+g(L u,v)
\end{align*}
The result now follows from the divergence theorem.
\end{proof}

Note that the above Lemma, the relationships between the Laplacians and the operator $L$ imply that for divergence free vector fields
\[
     \Big|\int_M\mathsf{Ri}(u,u)\omega_{M}\Big|\le \int_{M} g(\nabla u,\nabla u)\omega_{M}
\]
and
\begin{align*}
    & \int_M\mathsf{Ri}(u,u)\omega_{M}=0\quad\mathrm{if\ }u\mathrm{\ is\ parallel}   \\
    & \int_M\mathsf{Ri}(u,u)\omega_{M}=\int_{M} g(\nabla u,\nabla u)\omega_{M}\quad\mathrm{if\ }u\mathrm{\ is\ Killing}   \\
    & \int_M\mathsf{Ri}(u,u)\omega_{M}=-\int_{M} g(\nabla u,\nabla u)\omega_{M}\quad\mathrm{if\ }u\mathrm{\ is\ harmonic}
\end{align*}
So Killing vector fields and harmonic vector fields are at the ''opposite extremes'' with respect to curvature.

\section{Solutions to Navier-Stokes system and Killing fields}
Let us then start to analyze the properties of the solutions to \eqref{ns}. Let us first recall the following facts which are easy to check:
\begin{itemize}
\item[(i)] if $u$ is parallel and $p$ is constant then $(u,p)$ is a solution of Navier Stokes equations with Bochner Laplacian.
\item[(ii)] if $u$ is Killing and $p=\tfrac{1}{2}\,g(u,u)$ then $(u,p)$ is a solution of \eqref{ns}.
\item[(iii)] if $u$ is harmonic and $p=-\tfrac{1}{2}\,g(u,u)$ then $(u,p)$ is a solution of Navier Stokes equations with Hodge Laplacian.
\end{itemize}
Our first main result says that the component of any solution in the space of Killing fields remains constant.
\begin{theorem} Let $u$ be a solution of \eqref{ns} and $v$ be Killing. Then
\[
  \frac{d}{dt}\langle u,v\rangle=0
\]
\label{energia-1}
\end{theorem}
\begin{proof} First
\[
   \frac{d}{dt}\langle u,v\rangle= \langle u_t,v\rangle=
   -\langle \nabla_u u,v\rangle+\mu  \langle L  u,v\rangle-\langle \mathsf{grad}(p),v\rangle
\]
Then $\langle \nabla_u u,v\rangle=0$ by Lemma \ref{d-lemma}, $\langle L u,v\rangle=\langle u, Lv\rangle=0$ by Lemma \ref{var-lemma} and because $v$ is Killing, and $\langle \mathsf{grad}(p),v\rangle=-\langle p,\mathsf{div}(v)\rangle=0$ because $\mathsf{div}(v)=0$.
\end{proof}

In other words any solution can be decomposed as $u=u^K+u^\perp$ where $u^K$ is Killing and $u^\perp$ is orthogonal to Killing fields. One may view $u^K$ as a projection of the initial condition to the space of Killing fields. But then precisely with the same argument we get
\begin{theorem} Let $u$ be a solution of Navier Stokes equations with Hodge Laplacian and let $v$ be harmonic. Then
\[
     \frac{d}{dt}\langle u,v\rangle=0
\]
\label{energia-2}
\end{theorem}
\begin{proof} Now we have
\[
   \frac{d}{dt}\langle u,v\rangle= \langle u_t,v\rangle=
   -\langle \nabla_u u,v\rangle+\mu  \langle \Delta_H  u,v\rangle-\langle \mathsf{grad}(p),v\rangle
\]
Evidently $\langle \Delta_H u,v\rangle=\langle u, \Delta_Hv\rangle=0$ and
 $\langle \nabla_u u,v\rangle=0$ by Lemma \ref{d-lemma}.
\end{proof}

Let us then continue with system \eqref{ns}. The whole dynamics of the solution $u=u^K+u^\perp$ thus  happens in the component $u^\perp$. Then writing $p=p_K+p_\perp$ where $p_K=\tfrac{1}{2}\,g(u^K,u^K)$  we get the following system for $u^\perp$:
\begin{equation}
\begin{aligned}
  & u_t^\perp+\nabla_{u^\perp} u^K+\nabla_{u^K} u^\perp+\nabla_{u^\perp} u^\perp-\mu \, L u^\perp+ \mathsf{grad}( p_\perp)=0  \\
  &   \mathsf{div}( u^\perp)=0
\end{aligned}
\label{ns-perp}
\end{equation}
In the absence of forces acting on the system one expects that $u^\perp$ would approach zero when $t\to \infty$. To state this precisely we need a short digression.
Let us first define
\begin{align*}
         V_P=&\{ u\in H^1(M)\,|\,  \|u\|=1\ ,\  \langle u,v\rangle=0\ \mathrm{for\ all\ parallel}\ v\}  \\
           V_K=&\{ u\in H^1(M)\,|\,    \|u\|=1\ ,\   \langle u,v\rangle=0\ \mathrm{for\ all\ Killing}\ v\}  \\
            V_H=&\{ u\in H^1(M)\,|\,  \|u\|=1\ ,\  \langle u,v\rangle=0\ \mathrm{for\ all\ harmonic}\ v\}
\end{align*}
Then we can set
\begin{align*}
  \alpha_P=&\inf_{u\in V_P} \int_M g(\nabla u,\nabla u)\omega_{M} \\
   \alpha_K=&\inf_{u\in V_K} \int_M g(S_u,S_u)\omega_{M}  \\
      \alpha_H=&\inf_{u\in V_H} \int_M \Big(g(A_u,A_u)+\mathsf{div}(u)^2\Big)\omega_{M}
\end{align*}
What are the values of these constants? We could not find anything in the literature. The book \cite{hebey} treats extensively topics which are directly related, but everything is about functions, not vector fields. Anyway let us show that these numbers actually are positive.
\begin{theorem} The numbers $\alpha_P$, $\alpha_K$ and $\alpha_H$ are strictly positive and we have  Poincar\'e type inequalities:
\begin{align*}
    &  \alpha_P \int_M g(u,u)\omega_{M}\le  \int_M g(\nabla u,\nabla u)\omega_{M}\quad,\quad \forall\ u\in V_P    \\
    &  \alpha_K \int_M g(u,u)\omega_{M}\le  \int_M g(S_u,S_u)\omega_{M}\quad,\quad \forall\ u\in V_K    \\
    &\alpha_H\int_M g(u,u)\omega_{M}\le \int_M \Big(g(A_u,A_u)+\mathsf{div}(u)^2\Big)\omega_{M}
    \quad,\quad \forall\ u\in V_H
\end{align*}
\label{poincare-type}
\end{theorem}
\begin{proof} We adapt the idea of the proof of Theorem 2.10 in \cite{hebey} to the present context. All cases are essentially the same so for definiteness let us consider just the case $\alpha_P$. Let $u_k$ be a sequence such that
\[
  \lim_{k\to\infty} \int_M g(\nabla u_k,\nabla u_k)\omega_{M}=\inf_{u\in V_P} \int_M g(\nabla u,\nabla u)\omega_{M}
\]
By Rellich-Kondrakov Theorem there is a subsequence (still denoted by $u_k$) such that $u_k$ converges weakly in $H^1(M)$ and strongly in $L^2(M)$. Strong convergence implies that the limit $\hat u\in V_P$ and the weak convergence that
\[
   \int_M g(\nabla \hat u,\nabla \hat u)\omega_{M}\le  \lim_{k\to\infty} \int_M g(\nabla u_k,\nabla u_k)\omega_{M}
\]
Since $\int_M g(\nabla \hat u,\nabla \hat u)\omega_{M}>0$, $\alpha_P>0$.
\end{proof}

As a consequence we obtain
\begin{theorem} Let $u^\perp$ be a solution of \eqref{ns-perp}. Then
\[
   \|u^\perp\|^2\le C e^{-\mu \alpha_K t}
\]
\label{energia-perp}
\end{theorem}
\begin{proof}   Lemma \ref{b-lemma} and formula  \eqref{killing-ehto} imply that
  \[
   \int_{M}g(\nabla_{u^\perp} u^K,u^\perp)\omega_{M}=
    \int_{M}g(\nabla_{u^K} u^\perp, u^\perp)\omega_{M}=
     \int_{M}g(\nabla_{u^\perp} u^\perp, u^\perp)\omega_{M}=0
  \]
Then combining Theorem \ref{poincare-type} and Lemma \ref{var-lemma} we have
\[
   \frac{d}{dt}\, \|u^\perp\|^2=-\mu \int_M g(S_{u^\perp},S_{u^\perp})\omega_{M}\le
   -\mu \alpha_K \int_M g(u^\perp,u^\perp)\omega_{M}= -\mu \alpha_K  \|u^\perp\|^2
\]
\end{proof}

In particular the solutions of the stationary problem
\begin{align*}
  &\nabla_u u-\mu  L u+\mathsf{grad}( p)=0  \\
  &  \mathsf{div}(u)=0
\end{align*}
are precisely the Killing vector fields.

As a consequence we see that the asymptotic behavior of solutions is totally different for $L$ and $\Delta_H$. For example there are no harmonic vector fields on the sphere so that in the absence of forces all solutions tend to zero if Hodge Laplacian is used. But with the system \eqref{ns} the solutions tend to some Killing field. But the Killing fields on the sphere correspond to the rotating motion which is physically very natural. We think that this is one more argument in favor of $L$ compared to $\Delta_H$, in addition to the discussion in \cite{chczdi}.

At least for numerical purposes it is essential to analyze the elliptic equation satisfied by the pressure.
\begin{lemma}
\[
   -\Delta p-\mathsf{tr}((\nabla u)^2)-\mathsf{Ri}(u,u)+2\, \mu\,\mathsf{div}(\mathsf{Ri}(u))=0
\]
\label{pressure-equation}
\end{lemma}
\begin{proof}
First $\mathsf{div}\big(\Delta_B u\big)=g^{jk}u^i_{;jki}$; then applying the formula \eqref{ricci-twice} we get
\[
  u^i_{;jki}=u^i_{;ijk}+u^{\ell}_{;k}\mathsf{Ri}_{\ell j}
  +u^{\ell}\mathsf{Ri}_{\ell j;k}-u^i_{;\ell}R^{\ell}_{jik}+u^{\ell}_{;j}\mathsf{Ri}_{\ell k}
\]
But $g^{jk}\big(u^{\ell}_{;k}\mathsf{Ri}_{\ell j}-u^{i}_{;\ell}R^{\ell}_{jik} \big)=0$ which implies that
\begin{align*}
   \mathsf{div}\big(\Delta_B u\big)=&g^{jk}u^i_{;ijk}+g^{jk}u^{\ell}\mathsf{Ri}_{\ell j;k}+u^{\ell}_{;j}\mathsf{Ri}^{j}_{\ell}
     =g^{jk}u^i_{;ijk}+u^{\ell}\mathsf{Ri}_{\ell;k}^k+u^{\ell}_{;j}\mathsf{Ri}^{j}_{\ell}\\
     =&\Delta(\mathsf{div} (u))+\mathsf{div}(\mathsf{Ri} (u))
\end{align*}
Consequently
\begin{equation}
     \mathsf{div}( Lu)= 2\Delta (\mathsf{div}(u))+2\, \mathsf{div}(\mathsf{Ri}(u))
\label{divriu}
\end{equation}
Then we compute
\begin{align*}
    \mathsf{div}\big(\nabla_u u\big)=&u^k_{;ik}u^i+u^k_{;i}u^i_{;k}=u^k_{;ki}u^i+\mathsf{Ri}_{ij}u^iu^j+u^k_{;i}u^i_{;k}\\
                                    =&g(\mathsf{grad}(\mathsf{div}(u)), u) +\mathsf{Ri}(u,u)+\mathsf{tr}((\nabla u)^2)
\end{align*}
where we have used the formula \eqref{general-ricci-identity}. Then using the fact that $\mathsf{div}(u)=0$ gives the result.
\end{proof}

Note that the formula \eqref{divriu} implies that  $\mathsf{div}(\mathsf{Ri}(u))=0$  for Killing vector fields.  If Hodge Laplacian is used the same computations give
\[
   -\Delta p-\mathsf{tr}((\nabla u)^2)-\mathsf{Ri}(u,u)=0
\]
Hence there is no term which depends on the diffusion.
So for manifolds where the term $\mathsf{div}(\mathsf{Ri}(u))$ is big, for example manifolds whose curvature changes fast, the pressure given by $L$ and $\Delta_H$ should be considerably different.

\section{Linearized Navier Stokes}
In the numerical solution of Navier Stokes equations one has to deal with the linearized system so let us consider this case also.
\begin{equation}
\begin{aligned}
  & u_t+\nabla_u v+\nabla_v u-\mu  L u+\mathsf{grad}( p)=0  \\
  &  \mathsf{div}(u)=0
\end{aligned}
\label{ns-lin}
\end{equation}
Typically one can think of $v$ as the initial condition, and then one solves the same linear system for a few time steps. Interestingly the solution to the linearized system also preserves its aspect with respect to the space of Killing vector fields.
\begin{theorem}  Let $u$ be a solution of \eqref{ns-lin} and $w$ be Killing. Then
\[
  \frac{d}{dt}\langle u,w\rangle=0
\]
\label{energia-lin}
\end{theorem}
\begin{proof} As in Theorem \ref{energia-1} we first obtain
\[
   \frac{d}{dt}\langle u,w\rangle=
   -\langle \nabla_u v,w\rangle  -\langle \nabla_v u,w\rangle+\mu  \langle L  u,w\rangle-\langle \mathsf{grad}(p),w\rangle
\]
Again $\langle L u,w\rangle=\langle u, Lw\rangle=0$ and $\langle \mathsf{grad}(p),w\rangle=-\langle p,\mathsf{div}(w)\rangle=0$ because $w$ is Killing. But then
\[
     \langle \nabla_u v,w\rangle  +\langle \nabla_v u,w\rangle=0
\]
 using Lemma \ref{c-lemma}.
\end{proof}

This result is important because quite often in practise one uses either $\nabla_v u$ or $\nabla_u v$ for the linearized convection term because this is easier to implement. However, in that case the inner product is not preserved so that the computed solution is not qualitatively correct. This strongly suggests that a better solution is obtained when the full linearized convection term is used.

Let us then write as before $u=u^K+u^\perp$ and $v=v^K+v^\perp$. The pressure term can now be decomposed as $p=p_K+p_\perp=g(u,v)+p_\perp$. Let us further denote $f=-\nabla_{u^K} v^\perp-\nabla_{v^\perp} u^K$. Then we can write the linear system for $u^\perp$ as follows:
\begin{equation}
\begin{aligned}
  & u_t^\perp+\nabla_{v^K} u^\perp+\nabla_{u^\perp} v^K+\nabla_{u^\perp} v^\perp+\nabla_{v^\perp} u^\perp-\mu  L u^\perp+\mathsf{grad}(p_\perp)=f   \\
  &  \mathsf{div}(u^\perp)=0
\end{aligned}
\label{ns-lin-perp}
\end{equation}
In this case the norm of the solution does not necessarily diminish because now $f$ acts like a forcing term.
\begin{theorem}  Let $u^\perp$ be a solution of \eqref{ns-lin-perp}. Then
\begin{align*}
   \frac{d}{dt}\|u^\perp\|^2 &\le  -\mu \alpha_K  \|u^\perp\|^2
    +2\langle f,u^\perp\rangle
      -2\langle \nabla_{u^\perp} v^\perp,u^\perp\rangle
\end{align*}
\label{energia-lin-perp}
\end{theorem}
\begin{proof} First we have
\begin{align*}
  \tfrac{1}{2}\, \frac{d}{dt}\langle u^\perp,u^\perp\rangle=&
   \langle f,u^\perp\rangle
      -\langle \nabla_{v^K} u^\perp,u^\perp\rangle
   -\langle \nabla_{u^\perp} v^K,u^\perp\rangle\\
    &  -\langle \nabla_{u^\perp} v^\perp,u^\perp\rangle
   -\langle \nabla_{v^\perp} u^\perp,u^\perp\rangle
   +\mu  \langle L  u^\perp,u^\perp\rangle
   -\langle \mathsf{grad}(p),u^\perp\rangle
\end{align*}
First $\langle \mathsf{grad}(p),u^\perp\rangle=-\langle p ,\mathsf{div}(u^\perp)\rangle=0$ as usual. Now $g(\nabla_{u^\perp} v^K,u^\perp)=0$ because $v^K$ is Killing and $\langle \nabla_{v^K} u^\perp,u^\perp\rangle=\langle \nabla_{v^\perp} u^\perp,u^\perp\rangle=0$ by
Lemma \ref{b-lemma}. Hence
\[
    \tfrac{1}{2}\, \frac{d}{dt}\langle u^\perp,u^\perp\rangle=
   \langle f,u^\perp\rangle
      -\langle \nabla_{u^\perp} v^\perp,u^\perp\rangle
   +\mu  \langle L  u^\perp,u^\perp\rangle
\]
Then combining Theorem \ref{poincare-type} and Lemma \ref{var-lemma} we get the result.
\end{proof}

Hence in this case the norm might grow for some choices of $v$.
Note, however, that in the intended application the norm actually diminishes for small time. Typically $v$ is considered as the initial condition for $u$ so that $u(0)=u^K+u^\perp(0)=v^K+v^\perp$. Hence if we write
\[
    \beta(t)=\langle f,u^\perp\rangle-\langle \nabla_{u^\perp} v^\perp,u^\perp\rangle
\]
then clearly $\beta(0)=0$. But this stability for small time is sufficient in practice because one only takes a few time steps with same $v$.

\section{New solutions from old}
One interesting property of the linearized problem is that one can produce new solutions with the bracket. In the proof of the result we  need the following property of Killing vector fields \cite{petersen}.
If $v$ is a Killing vector field then
\[
    \nabla_{X,Y}^2v=R(Y,v)X
\]
which can be written with indices as
\begin{equation}
      v^{i}_{;hk}=-v^{\ell}R^{i}_{\ell kh}
\label{kill-curvature}
\end{equation}
\begin{theorem}  Let $(u,p)$ be a solution of \eqref{ns-lin} where we suppose that $v$ is Killing. Then
\[
  \big(\hat u,\hat p\big)=\Big([u,v], -g\big(\mathsf{grad}(p),v\big)\Big)
\]
is also a solution of \eqref{ns-lin}.
\label{sol-lin}
\end{theorem}
\begin{proof} First it is straightforward to check that $\mathsf{div}(\hat u)=\mathsf{div}([u,v])=0$.  For $\hat u$ we have the following equation:
\[
    \hat u_t+[\nabla_v u,v]+[\nabla_u v,v]-[\Delta_B u,v]-[\mathsf{Ri}(u),v]+[\mathsf{grad}(p),v]=0
\]
We have to verify the following claims.

\noindent\textbf{Claim 1}. $[\mathsf{grad}(p),v]=\mathsf{grad}(\hat p)$.
\begin{align*}
   [\mathsf{grad}(p),v]&=\nabla_{\mathsf{grad}(p)}v-\nabla_{v}(\mathsf{grad}(p))   \\
   &=   g^{k\ell}p_{;k}v^{i}_{;\ell}-g^{i\ell }p_{;k\ell}v^k=-g^{i\ell}p_{;k}v^k_{;\ell}-g^{i\ell }p_{;k\ell}v^k \\
   &=-g^{i\ell}\big( p_{;k}v^k\big)_{;\ell}
   =-\mathsf{grad}\big(g\big(\mathsf{grad}(p),v\big)\big)=\mathsf{grad}(\hat p)
\end{align*}
\noindent\textbf{Claim 2}.  $\nabla_v[u,v]=[\nabla_v u,v]$.

\noindent We have
\begin{align*}
   \nabla_{v}[u,v]&=\nabla_{v}\nabla_{u}v-\nabla_{v}\nabla_{v}u   \\
   &=\nabla_{(\nabla_{v}u)}v+v^{k}u^{h}v^{i}_{;hk}
   -\nabla_{v}(\nabla_{v}u)   \\
   &=[\nabla_{v}u,v] +v^{k}u^{h}v^{i}_{;hk}
\end{align*}
However, formula \eqref{kill-curvature} implies that
\[
   v^{k}u^{h}v^{i}_{;hk}=\nabla_{u,v}^2 v =R(v,v)u=0
\]

\noindent\textbf{Claim 3}.  $ \nabla_{[u,v]}v=[\nabla_u v,v]$.

\noindent Using the previous claim we compute
\begin{align*}
      \nabla_{[u,v]}v=& \nabla_v [u,v]+[ [u,v],v]=[\nabla_v u,v]+[ [u,v],v]  \\
                    =& [\nabla_u v+[v,u],v]+[ [u,v],v]=[\nabla_u v,v]
\end{align*}

\noindent\textbf{Claim 3}.  $ \Delta_{B}[u,v]=[\Delta_{B}u,v]$.

\noindent In coordinates  we have
\begin{align*}
  [\Delta_{B}u,v]=&g^{hk}\Big(u^{i}_{;hk}v^{\ell}_{;i}-v^{i}u^{\ell}_{;hki}\Big)  \\
    \Delta_{B}[u,v]=&g^{hk}\Big(u^{i}_{;hk}v^{\ell}_{;i}+2u^{i}_{;h}v^{\ell}_{;ik}+u^{i}v^{\ell}_{;ihk}
    -v^{i}_{;hk}u^{\ell}_{;i}-2v^{i}_{;h}u^{\ell}_{;ik}-v^{i}u^{\ell}_{;ihk}\Big)
\end{align*}
Let $T=\Delta_{B}[u,v]- [\Delta_{B}u,v]$; we will show that $T=0$. Using the formula \eqref{ricci-twice} we obtain
\begin{align*}
    T=g^{hk}\Big(2u^{i}_{;h}v^{\ell}_{;ik}+u^{i}v^{\ell}_{;ihk}
    -v^{i}_{;hk}u^{\ell}_{;i}-2v^{i}_{;h}u^{\ell}_{;ik}
    -v^i u^{\ell}_{;m}R^{m}_{ikh}+2v^iu^{j}_{;h}R^{\ell}_{ikj}+v^iu^{j}R^{\ell}_{ihj;k}  \Big)
\end{align*}
Since $\Delta_B v+\mathsf{Ri}(v)=0$ we have
\[
   g^{hk}v^{i}_{;hk}u^{\ell}_{;i}+ g^{hk}v^i u^{\ell}_{;m}R^{m}_{ikh}=
   g^{hk}v^{i}_{;hk}u^{\ell}_{;i}+ v^i u^{\ell}_{;m}\mathsf{Ri}^{m}_{i}=
   \nabla_{\Delta_B v}u+\nabla_{\mathsf{Ri}(v)}u=0
\]
Using this and formula \eqref{kill-curvature} thus yields
\[
    T=-g^{hk}\Big(u^i v^{m}_{;k}R^{\ell}_{mhi}+2v^{i}_{;h}u^{\ell}_{;ik} \Big)
\]
Then we can write
\[
    2v^{i}_{;h}u^{\ell}_{;ik}=v^{i}_{;h}u^{\ell}_{;ik}+v^i_{;h}u^{\ell}_{;ki}-v^i_{;h}u^{j}R^{\ell}_{ikj}
\]
so finally using the Killing property
\[
   T=-g^{hk}\Big(v^i_{;h}u^{\ell}_{;ki} +v^{i}_{;h}u^{\ell}_{;ik} \Big)=
    g^{ih}  v^k_{;h}u^{\ell}_{;ik}-g^{hk}v^{i}_{;h}u^{\ell}_{;ki} =0
\]

\noindent\textbf{Claim 4}.  $ \mathsf{Ri}[u,v]=[\mathsf{Ri}(u),v]$.

\noindent First we compute
\begin{align*}
  [\mathsf{Ri}(u),v]
  &=\mathsf{Ri}^{j}_{h}u^{h}v^{i}_{;j}-\mathsf{Ri}^{i}_{h;j}u^{h}v^{j}-\mathsf{Ri}^{i}_{h}u^{h}_{;j}v^{j} \\
  &=\mathsf{Ri}[u,v]+\mathsf{Ri}^{j}_{h}u^{h}v^{i}_{;j}-\mathsf{Ri}^{i}_{h;j}u^{h}v^{j}-\mathsf{Ri}^{i}_{h}u^{j}v^{h}_{;j}
\end{align*}
Hence we need to prove that
\begin{equation}
     T=\mathsf{Ri}^{j}_{h}v^{i}_{;j}-\mathsf{Ri}^{i}_{h;j}v^{j}-\mathsf{Ri}^{i}_{j}v^{j}_{;h}=0
\label{Ri-formula}
\end{equation}
Now formula \eqref{kill-curvature} implies that
\[
     v^{j}_{;\ell jh}=v^{j}_{;h}\mathsf{Ri}_{j \ell}+v^{j}\mathsf{Ri}_{j \ell;h}
\]
Applying this to the  formula  \eqref{Ri-formula} we get
\[
     T=\mathsf{Ri}^{j}_{h}v^{i}_{;j}-\mathsf{Ri}^{i}_{h;j}v^{j}+\mathsf{Ri}^{i}_{j;h}v^{j}-g^{i\ell} v^{j}_{;\ell jh}
\]
Using  Ricci identity  \eqref{general-ricci-identity}  and the fact that $v$ is Killing we get
\begin{align*}
    g^{i\ell} v^{j}_{;\ell jh}= & g^{i\ell}v^{j}_{;\ell hj}+ \mathsf{Ri}_{h}^\ell v^i_{;\ell}+ g^{i\ell}v^{j}_{;k}R^{k}_{jh\ell} \\
    =& - g^{i\ell}v^{k}_{;j} R^{j}_{kh\ell} - g^{i\ell}v^{k} R^{j}_{kh\ell;j} + \mathsf{Ri}_{h}^\ell v^i_{;\ell}+ g^{i\ell}v^{j}_{;k}R^{k}_{jh\ell} \\
    =& - g^{i\ell}v^{k} R^{j}_{kh\ell;j} + \mathsf{Ri}_{h}^\ell v^i_{;\ell}
\end{align*}
Then by Bianchi's second  identity \eqref{bianchi1}
\[
    T=\mathsf{Ri}^{i}_{j;h}v^{j}-\mathsf{Ri}^{i}_{h;j}v^{j}+ g^{i\ell}v^{k} R^{j}_{kh\ell;j}=
    \mathsf{Ri}^{i}_{j;h}v^{j}-\mathsf{Ri}^{i}_{h;j}v^{j}+
    v^{k} \big(\mathsf{Ri}^{i}_{h;k}-\mathsf{Ri}^{i}_{k;h}\big)=0
\]
\end{proof}

\section{2 dimensional case}
Since one of the main motivations for studying flows on manifolds comes from atmospheric models it is interesting to see this case in more detail. Moreover one has to take into account the Coriolis effect. Let us start, however,  with the arbitrary 2 dimensional manifold. The main simplification comes from the fact that in this case $\mathsf{Ri}=\kappa\, g$ where $\kappa$ is the Gaussian curvature. So the system can be written as follows
\begin{equation}
\begin{aligned}
  & u_t+\nabla_u u-\mu\Delta_B u-\mu \kappa\, u+\mathsf{grad}( p)=0   \\
  & -\Delta p-\mathsf{tr}((\nabla u)^2)-\kappa\, g(u,u)+2\, \mu\,g(\mathsf{grad}(\kappa), u)=0   \\
  &  \mathsf{div}(u)=0
\end{aligned}
\label{ns-2d}
\end{equation}
The Killing vector fields now satisfy the condition $g(\mathsf{grad}(\kappa), u)=0$; i.e. orbits defined by $u$ are on the level sets of the curvature. This makes intuitively clear the classical result about existence of Killing fields. Namely locally on 2 dimensional manifolds the space of Killing fields is either three, one or zero dimensional. If $\kappa$ is constant then we have the three dimensional case. If not the only solution candidates are vector fields which satisfy $g(\mathsf{grad}(\kappa), u)=0$. However, this is only necessary condition so depending on $\kappa$ the space can be zero or one dimensional. Note that globally the space of Killing fields can be two dimensional as the flat torus shows.

Since the vorticity is important in most of the fluid problems let us examine how it is in our context.  Recall that  the vorticity is  $\zeta=\mathsf{rot}(u)$ and using the formulas in \ref{curl-ja-muut} we can write
\[
    \zeta=\mathsf{rot}(u)=\mathsf{div}(Ku)=\varepsilon^{i}_{\ell}u^{\ell}_{;i}
\]
\begin{theorem} If $u$ is the solution of \eqref{ns-2d}  then
\begin{equation}
    \zeta_t-\mu\Delta \zeta+g(\mathsf{grad}(\zeta),u)-2\mu\,g(\mathsf{grad}(\kappa),Ku) -2\mu\,\kappa\zeta= 0
\label{vorti}
\end{equation}
\end{theorem}
\begin{proof}
In 2 dimensional case
\[
   Lu=\Delta_B u+\mathsf{grad}\big(\mathsf{div}(u)\big)+\kappa\,u
\]
and by the definition of $\mathsf{rot}$ it follows that $\mathsf{rot}\circ \mathsf{grad}=0$.
Now
\[
\mathsf{rot}\big(\kappa\,u\big)=\mathsf{div}\big(\kappa\,Ku\big)=
\kappa\,\zeta +g\big(\mathsf{grad}(\kappa),Ku\big)
\]
Then we compute
\begin{align*}
   \mathsf{rot}\big(\Delta_B u\big)=&g^{ij}\varepsilon^{\ell}_k u^k_{;ij\ell }  \\
     \Delta \zeta=&\Delta\big(\varepsilon^\ell_k u^k_{;\ell}\big)
      =g^{ij}\varepsilon^{\ell}_k u^k_{;\ell ij}
\end{align*}
Now using the formulas \eqref{curva2d} and \eqref{ricci-twice} we obtain
\begin{align*}
  \Delta \zeta=& g^{ij}\varepsilon^{\ell}_k u^k_{;\ell ij}=g^{ij}\varepsilon^{\ell}_k u^k_{;ij\ell}
    -\kappa\,   u^{h}_{;i}\varepsilon^i_h-  \kappa_{;i}  u^{h}\varepsilon^i_h   \\
             =& \mathsf{rot}\big(\Delta_B u\big)-\kappa\zeta
    -g\big(Ku,\mathsf{grad}(\kappa)\big)
\end{align*}
Then using the Ricci identity and the formula \eqref{curva2d} we get
\begin{align*}
    \mathsf{rot}(\nabla_{u} u)&=\varepsilon^i_{\ell}u^{j}_{;i}u^{\ell}_{;j}+\varepsilon^i_{\ell}u^{j}u^{\ell}_{;ji} \\
    &=\varepsilon^i_{\ell}u^{j}_{;i}u^{\ell}_{;j}+\varepsilon^i_{\ell}u^{j}(u^{\ell}_{;ij}-u^{h}R^{\ell}_{jih}) \\
    &=\varepsilon^i_{\ell}u^{j}_{;i}u^{\ell}_{;j}+\zeta_{;j}u^{j}-
   \kappa\, \varepsilon_{jh}u^{j}u^{h}
\end{align*}
But $\varepsilon_{jh}u^ju^h=0$ and a direct computation shows that $\varepsilon^i_{\ell}u^{\ell}_{;j}u^{j}_{;i}=\varepsilon^i_{\ell}u^{\ell}_{;i}u^{j}_{;j}=\zeta\,\mathsf{div}(u)$.
\end{proof}

Since the sphere is an important special case let us analyze this case more closely.  Let us first recall the following result:
\begin{itemize}
\item[] let $f$ be a function on $M$ such that $\int_Mf\omega_{M}=0$; then
\begin{equation}
   \lambda_1\int_M f^2\omega_{M}\le \int_M g(\mathsf{grad}(f),\mathsf{grad}(f))\omega_{M}
\label{omi}
\end{equation}
where $\lambda_1$ is the first positive eigenvalue of $-\Delta$. More briefly we can write $\lambda_1\|f\|^2\le \|\mathsf{grad}(f)\|^2$.
\end{itemize}
This gives
\begin{theorem} Let $\zeta$ be the solution of \eqref{vorti} on the sphere; then
\[
   \tfrac{d}{dt} \|\zeta\|^2\le0
\]
\label{vorti-teo}
\end{theorem}
\begin{proof} First
\[
   \tfrac{1}{2}  \tfrac{d}{dt} \|\zeta\|^2=\mu\int_M\zeta \Delta \zeta\omega_{M}-
   \int_M g(\mathsf{grad}(\zeta),u)\zeta\omega_{M}+
      2\mu\int_M \kappa\zeta^2\omega_{M}
\]
The formula
\[
 \mathsf{div}\big(\tfrac{1}{2}\zeta^2 u\big)=\tfrac{1}{2}\zeta^2\mathsf{div}(u)
 +g(\mathsf{grad}(\zeta),u)\zeta
\]
implies that $ \int_M g(\mathsf{grad}(\zeta),u)\zeta\omega_{M}=0$. The result then follows form the inequality \eqref{omi} because $ \int_M \zeta\omega_{M}=0$ and on the sphere $\lambda_1=2\kappa$.
\end{proof}

Let us again use
 the decomposition $u=u^K+u^\perp$ for the solution of \eqref{ns} and let  $\zeta=\zeta^K+\zeta^\perp$ be the corresponding decomposition for the vorticity. Note that for all 2 dimensional manifolds we have
 \[
   \mathsf{Rot}(\mathsf{rot}(u^K))=
          2\kappa u^K
 \]
 Multiplying by $K$ and taking the inner product with $u^K$ gives
 \[
        g(\mathsf{grad}(\zeta^K),u^K)=-2\kappa g(Ku^K,u^K)=0
 \]
 On the other hand applying the operator $\mathsf{rot}$ implies that on the sphere
 \[
     -\Delta \zeta^K=2\,\kappa\zeta^K
\]
Hence  the functions $\zeta^K$ are actually the first spherical harmonics, i.e. the eigenfunctions of $-\Delta$ corresponding to the smallest positive eigenvalue.

Interestingly the orthogonality  of $u^K$ and $u^\perp$ ''descends'' to the orthogonality of vorticities.
\begin{lemma} With the above notations on the sphere we have $\langle \zeta^K,\zeta^\perp\rangle=0$.
\end{lemma}
\begin{proof} We compute
\begin{align*}
   \int_M  \zeta^K\zeta^\perp\omega_{M}=&
        \int_M\mathsf{rot}(u^K)\mathsf{rot}(u^\perp)\omega_{M}=
        -\int_M g\big(Ku^\perp,\mathsf{grad}(\mathsf{rot}(u^K))\big)\omega_{M}\\
        =&-\int_M g\big(u^\perp,\mathsf{Rot}(\mathsf{rot}(u^K))\big)\omega_{M}
        =-2\kappa\int_M g(u^\perp,u^K)\omega_{M}=0
\end{align*}
\end{proof}

 So on the sphere the dynamics of $\zeta$ happens on the component $\zeta^\perp$.  Then let us see what is the equation for $\zeta^\perp$. Substituting $u=u^K+u^\perp$ and  $\zeta=\zeta^K+\zeta^\perp$ to \eqref{vorti} and taking into account that (i) $-\Delta \zeta^K=2\,\kappa\zeta^K$, (ii) $g(\mathsf{grad}(\zeta^K),u^K)=0$ and (iii)  $g(\mathsf{grad}(\zeta^K),u^\perp)=2\kappa g(u^K,Ku^\perp)$ gives
 \[
     \zeta_t^\perp-\mu\Delta \zeta^\perp+2\kappa g(u^K,Ku^\perp)
     +g(\mathsf{grad}(\zeta^\perp),u^K)+g(\mathsf{grad}(\zeta^\perp),u^\perp) -2\mu\,\kappa\zeta^\perp= 0
 \]
This allows us to estimate more precisely the norm of $\zeta^\perp$. To this end we need the following
\begin{lemma}
For any vector fields $u$ and $v$ on a 2 dimensional manifold we have
\[
    \nabla_{(Ku)}v+\nabla_{u}(Kv)=\mathsf{div}(v)\,Ku+\mathsf{div}(Kv)\,u
\]
In particular
\[
    \nabla_{Ku}Ku-\mathsf{div}(Ku)\,Ku=\nabla_{u}u- \mathsf{div}(u)\,u
\]
\label{2d-lem}
\end{lemma}
\begin{proof}
First we have
\begin{align*}
    Ku&=g^{ih}\varepsilon_{hj}u^{j}=
    \varepsilon_{12}(g^{i1}u^{2}-g^{i2}u^{1})\\
     \nabla_{Ku}v&=g^{kh}\varepsilon_{hj}u^{j}v^{i}_{;k}=\varepsilon_{12}(g^{k1}u^{2}v^{i}_{;k}-g^{k2}u^{1}v^{i}_{;k})\\
       \nabla_{u}Kv&=g^{ih}\varepsilon_{hj}v^{j}_{;k}u^{k}=\varepsilon_{12}(g^{i1}v^{2}_{;k}u^{k}-g^{i2}v^{1}_{;k}u^{k})\\
       \mathsf{div}(Kv)&=g^{kh}\varepsilon_{hj}v^j_{;k}=
       \varepsilon_{12}(g^{k1}v^{2}_{;k}-g^{k2}v^{1}_{;k})
\end{align*}
We have to show that
\[
  w^i=\big(g^{k1}u^{2}-g^{k2}u^{1}\big)v^{i}_{;k}+
   \big( g^{i1}v^{2}_{;k}-g^{i2}v^{1}_{;k}\big)u^{k}
    -(g^{i1}u^{2}-g^{i2}u^{1})v^k_{;k}
    -(g^{k1}v^{2}_{;k}-g^{k2}v^{1}_{;k})u^i=0
\]
But simply expanding the components we can check that $w^1=w^2=0$.
\end{proof}

 Note that now we have shown that  $\zeta^\perp$ is orthogonal to the zeroth and first eigenspaces of $-\Delta$. But then by the minimum characterization of the eigenvalues this implies that
 \begin{equation}
   \lambda_2\int_M (\zeta^\perp)^2\omega_{M}=
    6\kappa \int_M (\zeta^\perp)^2\omega_{M}
   \le \int_M g(\mathsf{grad}(\zeta^\perp),\mathsf{grad}(\zeta^\perp))\omega_{M}
\label{omi2}
\end{equation}

 \begin{theorem} Let $\zeta^\perp$ be the solution of \eqref{vorti} on the sphere; then
\[
  \|\zeta^\perp\|^2\le Ce^{-8\mu\kappa t}
\]
\end{theorem}
\begin{proof} First
\begin{align*}
   \tfrac{1}{2}  \tfrac{d}{dt} \|\zeta^\perp\|^2=&\mu\int_M\zeta^\perp \Delta \zeta^\perp\omega_{M}-
   2\kappa\int_M  g(u^K,Ku^\perp)\zeta^\perp\omega_{M}\\
  - &\int_M g(\mathsf{grad}(\zeta^\perp),u^K)\zeta^\perp\omega_{M}-
   \int_M g(\mathsf{grad}(\zeta^\perp),u^\perp)\zeta^\perp\omega_{M}+
      2\mu\kappa\int_M (\zeta^\perp)^2\omega_{M}
\end{align*}
Then we have
\[
      \int_M g(\mathsf{grad}(\zeta^\perp),u^K)\zeta^\perp\omega_{M}=
   \int_M g(\mathsf{grad}(\zeta^\perp),u^\perp)\zeta^\perp\omega_{M}=0
\]
by the same argument as in the proof of Theorem \ref{vorti-teo}. Then we compute
\begin{align*}
      \mathsf{div}\big(g(u^K,Ku^\perp)Ku^\perp\big)=&
       g(u^K,Ku^\perp) \zeta^\perp+g\big(\nabla_{Ku^\perp}u^K,Ku^\perp\big)+
       g\big(\nabla_{Ku^\perp}Ku^\perp,u^K\big)\\
       =& 2g(u^K,Ku^\perp) \zeta^\perp+
       g\big(\nabla_{u^\perp}u^\perp,u^K\big)
\end{align*}
where the second equality holds because $u^K$ is Killing and  Lemma \ref{2d-lem}. Hence
\[
  \int_M  g(u^K,Ku^\perp)\zeta^\perp\omega_{M}=0
\]
by Lemma \ref{c-lemma}.
Then the inequality \eqref{omi2} gives
\begin{align*}
   \tfrac{1}{2}  \tfrac{d}{dt} \|\zeta^\perp\|^2=&
   -\mu\int_Mg(\mathsf{grad}(\zeta^\perp),\mathsf{grad}(\zeta^\perp))\omega_{M}-
      2\mu\kappa\int_M (\zeta^\perp)^2\omega_{M}\\
      \le& -8\mu\kappa \int_M (\zeta^\perp)^2\omega_{M}
\end{align*}
\end{proof}

\section{Coriolis}
Let us then consider the Coriolis effect. For the simplicty of notation let us assume that our manifold is now the unit sphere $S^2$.  The rotation of the sphere has two effects: centrifugal force and Coriolis force. Since the centrifugal force is conservative one can absorb it to the pressure. This modified pressure is still denoted by $p$. From the Coriolis force there comes a new term to the system which is of the form $a\,Ku$ where $a$ is some function. Let us indicate how to express $a$ in spherical coordinates $(\theta,\varphi)$ where $\theta$ is the longitude and $\varphi$ is the colatitude.  Let us interpret $S^2$ as a submanifold of $\mathbb{R}^3$.  If we now choose $x_3$ axis to be  the axis of rotation with the rotation vector $(0,0,\omega)$ then $a=2\omega\cos(\varphi)$.
The system can thus be written as
\begin{equation}
\begin{aligned}
  & u_t+\nabla_u u-\mu L u+a\,Ku+\mathsf{grad}( p)=0\\
 & -\Delta p-\mathsf{tr}((\nabla u)^2)-\, g(u,u)-\mathsf{div}\big(a\,Ku\big)=0\\
  &  \mathsf{div}(u)=0
 \end{aligned}
  \label{ns-2d-cor}
\end{equation}
Since $g(Ku,u)=0$ the Coriolis term has no effect on the norm: we still have
\[
      \frac{d}{dt}\, \|u\|^2=-\mu \int_M g(S_{u},S_{u})\omega_{M}
\]
However, not all Killing fields are now solutions. Let $u$ be Killing; if it is a solution to \eqref{ns-2d-cor} then we should have
\[
 \mathsf{rot}(aKu)=-a\,\mathsf{div}(u)-g(\mathsf{grad}(a),u)=-g(\mathsf{grad}(a),u)=0
\]
Hence in spherical coordinates $u=c\,\partial_\theta$ where $c$ is constant. Simple computations show that the corresponding pressure is
\[
    p_K=\tfrac{1}{2}\,\big(c^2\sin(\varphi)^2+c\,\omega \cos(2\varphi)\big)
\]
Let us thus denote by $(u^K,p_K)$ this solution which in spherical coordinates are given by above formulas.
Then we can also in this situation try to look for solutions of the form $u=u^K+\hat u$ and $p=p_K+\hat p$. Note that here we do not have the result like Theorem \ref{energia-1}. In spite of this it turns out that the energy of $\hat u$ decreases monotonically.
\begin{theorem}
Let $u=u^K+\hat u$, $p=p_K+\hat p$  be a solution to (\ref{ns-2d-cor}). Then
\[
    \tfrac{d}{dt}\|\hat u\|^2 \,\leq 0
\]
\end{theorem}
\begin{proof}
Computing as before we find the following system for $(\hat u,\hat p)$.
\begin{align*}
   & \hat u_{t}+\nabla_{\hat u} u^{K}+\nabla_{u^K} \hat u+\nabla_{\hat u} \hat u-\mu\,L \hat u+a\,K \hat u+\mathsf{grad}(\hat p)=0 \\
    & \mathsf{div}(\hat u)=0
\end{align*}
Then the variational formulation can be written as
\begin{align*}
    \tfrac{1}{2}\, \tfrac{d}{dt}\, \|\hat u\|^2&=-\int_{M}g(\nabla_{\hat u} u^{K},\hat u)\,\omega_{M}-\int_{M}g(\nabla_{u^{K}} \hat u,\hat u)\,\omega_{M}-\int_{M}g(\nabla_{\hat u} \hat u,\hat u) \,\omega_{M} \\
     &+\mu\,\int_{M}g\big(L\hat u,\hat u\big)\,\omega_{M}-\int_{M}a\,g\big(K\hat u,\hat u\big)\,\omega_{M}-\int_{M}g\big(\mathsf{grad}(\hat p),\hat u\big)\,\omega_{M}
\end{align*}
Then applying Lemma \ref{b-lemma}, Lemma \ref{d-lemma}, divergence theorem and since $g(Ku,u)=0$, we get
\[
    \tfrac{d}{dt}\, \|\hat u\|^2=-\mu \int_M g(S_{\hat u},S_{\hat u})\,\omega_{M}
\]
\end{proof}

Hence the Coriolis term tends to align the flow along the circles of latitude. Note finally that this conclusion depends on the choice of the diffusion operator. If the Hodge Laplacian is used then the solutions simply approach zero because on the sphere there are no harmonic vector fields.


\appendix

\section{Notation and some formulas }
Let us review some basic notions of Riemannian geometry. For details we refer to  \cite{petersen} and  \cite{spivak2}. For curvature tensor and Ricci tensor there are several different conventions regarding the indices and signs. We will follow the conventions in \cite{petersen}.

The curvature tensor is denoted by $R$ and Ricci tensor by $\mathsf{Ri}$. In coordinates we have
\begin{equation}
\begin{aligned}
     \mathsf{Ri}_{jk}&=R^i_{ijk}=g^{i\ell}R_{ij k\ell}  \\
      \mathsf{Ri}_j^k&=g^{k\ell}\mathsf{Ri}_{j\ell}=g^{k\ell}R^i_{ij\ell }=g^{i\ell}R^k_{ji\ell}
\end{aligned}
\label{ricci-tensori}
\end{equation}
The scalar curvature is $\mathsf{R_{sc}}=\mathsf{Ri}_k^k$. The Bianchi identities are
\begin{equation}
\begin{aligned}
    R^i_{jk\ell}+ R^i_{k\ell j}+ R^i_{\ell jk}=0   \\
    R_{hijk;\ell}+R_{hi\ell j;k}+R_{hik\ell;j}=0
\end{aligned}
\label{bianchi1}
\end{equation}
For general  tensors $A$ of type $(m, n)$ the \emph{Ricci identity} has the form
\begin{equation}
    A^{j_{1}\cdots j_{m}}_{i_{1}\cdots i_{n};ij}-A^{j_{1}\cdots j_{m}}_{i_{1}\cdots i_{n};ji}=\sum^{n}_{q=1}A^{j_{1}\cdots j_{m}}_{i_{1}\cdots i_{q-1}\ell i_{q+1}\cdots i_{n}}R^{\ell}_{iji_{q}}-\sum^{m}_{p=1}A^{j_{1}\cdots j_{p-1}\ell j_{p+1}\cdots j_{m}}_{i_{1}\cdots i_{n}}R^{j_{p}}_{ij\ell }
\label{general-ricci-identity}
\end{equation}
The following consequences where Ricci identity is used twice are used in many places
\begin{equation}
\begin{aligned}
   u^k_{;\ell ij}&=u^{k}_{;ij \ell }+u^{k}_{;h}R^{h}_{\ell j i}-u^{h}_{;i}R^{k}_{\ell j h}-u^h_{;j }R^{k}_{\ell ih}-u^hR^{k}_{\ell ih;j }  \\
   u^k_{;ki j}&=u^k_{;i jk}+u^k_{;h}R^h_{kji}-u^h_{;i}\mathsf{Ri}_{ jh}-u^{h}_{;j}\mathsf{Ri}_{ ih} -u^h\mathsf{Ri}_{h i;j}
\end{aligned}
\label{ricci-twice}
\end{equation}
In the two dimensional case the curvature tensor can be written as follows:
\begin{equation}
\begin{aligned}
     R_{ijk\ell}=&\kappa (g_{i\ell}g_{jk}-g_{ik}g_{j\ell})   \\
     R^\ell_{ijk}=&\kappa\big(g_{jk}\delta^\ell_i-g_{ik}\delta^\ell_j\big)   \\
    \varepsilon^j_\ell  R^\ell_{ijk}=&-\kappa\,\varepsilon_{ik}   \\
\end{aligned}
\label{curva2d}
\end{equation}
Here $\kappa$ is the Gaussian curvature.

\section{Operators $\mathsf{rot}$, $\mathsf{Rot}$ and $\mathsf{curl}$}
\label{curl-ja-muut}
Let us denote by $V=C^\infty(M)$ the space of smooth functions on $M$, let $\mathfrak{X}(M)$  be the space of vector fields and $\bigwedge^k\!M$ the space of $k$ forms.  Let us  suppose that $M$ is two dimensional.  Then we define the operator $\mathsf{rot}$ by requiring that the following diagram commutes.
\begin{equation}
\xymatrix{
0\ar[r]&V\ar[r]^-{\mathsf{grad}}\ar@{=}[d]&\mathfrak{X}(M)\ar[d]^-{\flat}\ar[r]^-{\mathsf{rot}}
  & V\ar[d]^-{\ast}\ar[r] &0\\
0\ar[r]&
V\ar[r]^-{d} &
\bigwedge^1\!M \ar[r]^-{d} & \bigwedge^2\!M \ar[r] &0
}
\label{rot-kaavio}
\end{equation}
Here $\flat$ is the usual map $T_pM\to T_p^\ast M$ defined by the Riemannian metric and $\ast$ is the Hodge operator. To express this in coordinates we first define
\[
  \varepsilon=\sqrt{\det(g)}\big(dx_1\otimes dx_2-dx_2\otimes dx_1\big)
\]
Note that $\nabla \varepsilon=0$. Then it is convenient to introduce the map
\begin{equation}
   Ku=g^{ki}\varepsilon_{ij}u^j=\varepsilon^k_{j}u^j
\label{K-def}
\end{equation}
 Intuitively the operator $K$  rotates the vector field by 90 degrees.
Then we can write
\[
    \mathsf{rot}(u)=\mathsf{div}(Ku)=\varepsilon^k_j u^j_{;k}
\]
We will also need the $\mathsf{Rot}$ operator which is defined by the following diagram
\begin{equation}
\xymatrix{
0\ar[r]&V\ar[r]^-{\mathsf{Rot}}\ar@{=}[d]
  &\mathfrak{X}(M)\ar[d]^-{\iota_\omega}\ar[r]^-{\mathsf{div}}
  & V\ar[d]^-{\ast}\ar[r] &0\\
0\ar[r]&
V\ar[r]^-{d} &
\bigwedge^1\!M \ar[r]^-{d} &  \bigwedge^2\!M \ar[r] &0
}
\label{Rot-kaavio}
\end{equation}
Here $\iota_\omega$ is the interior product. In coordinates we have
\[
     \mathsf{Rot}(u)=-K\,\mathsf{grad}(u)=- \varepsilon^k_i g^{ij}u_{;j}=-  \varepsilon^{kj} u_{;j}
\]
Let us now suppose that $M$ is three dimensional. Then we define the operator $\mathsf{curl}$ by requiring that the following diagram commutes.
\begin{equation}
\xymatrix{
0\ar[r]&V\ar[r]^-{\mathsf{grad}}\ar@{=}[d]&\mathfrak{X}(M)\ar[d]^-{\flat}\ar[r]^-{\mathsf{curl}}
  &\mathfrak{X}(M)\ar[d]^-{\iota_\omega}\ar[r]^-{\mathsf{div}}
  & V\ar[d]^-{\ast}\ar[r] &0\\
0\ar[r]&
V\ar[r]^-{d} &
\bigwedge^1\!M \ar[r]^-{d} & \bigwedge^2\!M \ar[r]^-{d}  & \bigwedge^3\!M \ar[r] &0
}
\label{curl-kaavio}
\end{equation}
Let us now define
\begin{align*}
\varepsilon=  \sqrt{\det(g)}\,\Big(&
    dx_1\otimes dx_2\otimes dx_3
    -dx_2\otimes dx_1\otimes dx_3
    -dx_3\otimes dx_2\otimes dx_1\\
   & -dx_1\otimes dx_3\otimes dx_2
    +dx_2\otimes dx_3\otimes dx_1
    +dx_3\otimes dx_1\otimes dx_2\Big)
\end{align*}
Again $\nabla \varepsilon=0$. Then we can express $\mathsf{curl}$ in coordinates by the formula
\[
    \mathsf{curl}(u)=\varepsilon^{ijk}g_{j\ell}u^\ell_{;i}
\]
We can also define the cross product of two vector fields by
\[
  u\times v=\sharp\,\ast (\flat u\wedge \flat v)=g^{\ell k}\varepsilon_{ijk} u^i v^j
\]

\printbibliography

\end{document}